\documentclass[11pt,a4paper]{article}
\usepackage[utf8]{inputenc}
\usepackage{epsfig,graphicx,graphics}
\usepackage{amsmath}
\usepackage{amssymb}

\usepackage{amsthm}
\usepackage{amsmath}
\usepackage{amsfonts}
\usepackage{geometry}
\usepackage{epsfig}

\usepackage{color}
\usepackage{amssymb}
\usepackage{graphicx}
\usepackage{amsmath,amsfonts,amsthm,amstext}
\usepackage{mathrsfs}
\usepackage{hyperref}

\newtheorem{teo}{Theorem}[section]
\newtheorem{lema}{Lemma}[section]
\newtheorem{cor}{Corollary}[section]
\newtheorem{propo}{Proposition}[section]
\theoremstyle{definition}
\newtheorem{defi}{Definition}[section]

\newtheorem{rek}{Remark}[section]

\newcommand{\supp}{\operatorname{supp}}

\newcommand{\diam}{\operatorname{diam}}

\newcommand{\fix}{\operatorname{Fix}}

\newcommand{\essinf}{\operatorname{ess\,inf}}

\newcommand{\esssup}{\operatorname{ess\,sup}}

\newcommand{\card}{\operatorname{card}}

\newenvironment{proof01}[1][\textit{\textbf{Proof \em(Theorem~\ref{torigual0})}}]{\textit{#1.} }{\hfill $\Box$}

\newenvironment{proof02}[1][\textit{\textbf{Proof \em(Corollary~\ref{quasiyoung})}}]{\textit{#1.} }{\hfill $\Box$}

\newenvironment{proof03}[1][\textit{\textbf{Proof \em(Theorem~\ref{BKpunctually})}}]{\textit{#1.} }{\hfill $\Box$}

\newenvironment{proof1.1}[1][\textit{\textbf{Proof \em(Theorem~\ref{MU})}}]{\textit{#1.} }{\hfill $\Box$}

\newenvironment{proof3}[1][\textit{\textbf{Proof \em(Theorem~\ref{allzerodimensions})}}]{\textit{#1.} }{\hfill $\Box$}

\newenvironment{proof4}[1][\textit{\textbf{Proof \em(Theorem~\ref{allzerodimensions1})}}]{\textit{#1.} }{\hfill $\Box$}

\newcommand{\ve}{\varepsilon}
\newcommand{\N}{\mathbb{N}}
\newcommand{\R}{\mathbb{R}}
\newcommand{\Z}{\mathbb{Z}}
\newcommand{\M}{\mathcal{M}}

\begin{document}
\bibliographystyle{plain}

\title {
A note on the relation between the  metric entropy and the generalized fractal dimensions of invariant measures}
\date{}

\author{Alexander Condori \thanks{ Work  partially  supported  by  CIENCIACTIVA C.G. 176-2015}~~and~ Silas L. Carvalho \thanks{Work partially supported by FAPEMIG (a Brazilian government agency; Universal Project 001/17/CEX-APQ-00352-17)}}
\maketitle

{\small \noindent $^{*}\,$Instituto de Matem\'atica y Ciencias Afines (IMCA-UNI). Calle Los Bi\'ologos 245.
Lima 15012
Per\'u. \\ {\em e-mail:
  a.condori@imca.edu.pe }}

{ \small \noindent $\dag\,$Instituto de Ciências Exatas (ICEX-UFMG). Av. Pres. Antônio Carlos 6627, Belo Horizonte-MG, 31270-901, Brasil. \\ {\em e-mail:
silas@mat.ufmg.br }}

\maketitle

\begin{abstract}
\noindent{
  We investigate in this work some situations where it is possible to estimate or determine the upper and the lower $q$-generalized fractal dimensions $D^{\pm}_{\mu}(q)$, $q\in\R$, of invariant measures associated with continuous transformations over compact metric spaces. In particular, we present an alternative proof of Young's Theorem~\cite{Young} for the generalized fractal dimensions of the Bowen-Margulis measure associated with a $C^{1+\alpha}$-Axiom A system over a two-dimensional compact Riemannian manifold $M$.  We also present estimates for the generalized fractal dimensions of an ergodic measure for which Brin-Katok's Theorem is satisfied punctually, in terms of its  metric entropy.

  Furthermore, for expansive homeomorphisms (like $C^1$-Axiom A systems), we show that the set of invariant measures such that $D_\mu^+(q)=0$ ($q\ge 1$), under a hyperbolic metric, is generic (taking into account the weak topology). We also show that for each $s\in[0,1)$, $D^{+}_{\mu}(s)$ is bounded above, up to a constant, by the topological entropy, also under a hyperbolic metric.

Finally, we show that, for some dynamical systems, the metric entropy of an invariant measure is typically zero, settling a conjecture posed by Sigmund in~\cite{Sigmund1974} for Lipschitz transformations which satisfy the specification property.
}
\noindent
\end{abstract}
\section{Introduction}

The dimension theory of invariant measures plays a very  important role in the theory of dynamical systems.

There are several different notions of dimension for more general sets, some easier to compute and others more convenient in applications. One of them is, and could be said to be the most popular of all, the Hausdorff dimension, introduced in 1919 by Hausdorff, which gives a notion of size useful for distinguishing between sets of zero Lebesgue measure. 

Unfortunately, the  Hausdorff dimension of relatively simple sets can be very hard to calculate; besides, the notion of Hausdorff dimension is not completely adapted to the dynamics per se (for instance, if $Z$ is a periodic orbit, then its Hausdorff  dimension is zero, regardless to whether the orbit is stable, unstable, or neutral). This fact led to the introduction of other characteristics for which it is possible to estimate the size of irregular sets. For this reason, some of these quantities were also branded as ``dimensions'' (although some of them lack some basic properties satisfied by Hausdorff dimension, such as $\sigma$-stability; see \cite{Mattila}). Several good candidates were proposed, such as the correlation, information, box counting and entropy dimensions, among others.

Thus, in order to obtain relevant information about the dynamics, one should consider not only the geometry of the measurable set $Z\subset X$ (where $X$ is some Borel measurable space), but also the distribution of points on $Z$ under $f$ (which is assumed to be a measurable transformation). That is, one should be interested in how often a given point $x \in Z$ visits a fixed subset $Y\subset Z$ under $f$. If $\mu$ is an ergodic measure for which $\mu(Y)>0$, then for a typical point $x\in Z$, the average number of visits is equal to $\mu(Y)$. Thus, the orbit distribution is completely determined by the measure $\mu$. On the other hand, the measure $\mu$ is completely specified by the distribution of a typical orbit.

This fact is widely used in the numerical study of dynamical systems where the distributions are, in general, non-uniform and have a clearly visible fine-scaled interwoven structure of hot and cold spots, that is, regions where the frequency of visitations is either much greater than average or much less than average respectively.

In this direction, the so-called correlation dimension of a probability measure was introduced by Grassberger, Procaccia and Hentschel \cite{Grassberger}  in an attempt to produce a characteristic of a dynamical system that captures information about the global behavior of typical (with respect to an invariant measure) trajectories by observing only one them.

This dimension plays an important role in the numerical investigation  of chaotic behavior in different models, including strange attractors. The formal definition  is as follows (see~\cite{PesinT1995,Pesin1993,Pesin1997}): let $(X,r)$ be a complete and separable (Polish) metric space, and let $f:X\rightarrow X$ be a continuous map. Given $x\in X$, $\ve>0$  and $n\in\N$, one defines the correlation sum of order $q\in\N\setminus\{1\}$ (specified by the points $\{f^i(x)\}$, $i=1,\ldots,n$) by
\begin{equation*}
C_q(x,n,\ve)=\frac{1}{n^q}\,\card\,\{(i_1\cdots i_q)\in \{0,1,\cdots, n\}^q\mid r(f^{i_j}(x),f^{i_l}(x))\leq \ve ~ \mbox{ for any } ~ 0\leq j,l\leq q\},
\end{equation*}
where $\card A$ is the cardinality of the set $A$. Given $x\in X$, one defines (when the limit $n\to\infty$ exists) the quantities 
\begin{equation}\label{cordim}
 \underline{\alpha}_q(x)=  \frac{1}{q-1}\lim_{\overline{ \ve\to 0}}   \lim_{n\to \infty}\frac{\log C_q(x,n,\ve)}{\log(\ve)}, ~~  ~~  \overline{\alpha}_q(x)= \frac{1}{q-1}\overline{\lim_{\ve\to 0}} \lim_{n\to \infty}\frac{\log C_q(x,n,\ve)}{\log(\ve)},
\end{equation}
the so-called \emph{lower} and \emph{upper correlation dimensions of order $q$ at the  point $x$} or the \emph{lower} and the \emph{upper $q$-correlation dimensions at $x$}. If the limit $\ve\to 0$ exists, we denote it by $\alpha_q$, the so-called \emph{$q$-correlation dimension at $x$}. In this case, if $n$ is large and $\ve$ is small, one has the asymptotic relation
\[C_q(x,n,\ve) \thicksim\ve^{\alpha_q}.\]
 
$C_q(x,n,\ve)$ gives an account of how the orbit of $x$, truncated at time $n$, ``folds'' into an $\ve$-neighborhood of itself; the larger $C_q(x,n,\ve)$, the ``tighter'' this truncated orbit is. $\underline{\alpha}_q(x)$ and $ \overline{\alpha}_q(x)$ are, respectively, the lower and upper growing rates of $C_q(x,n,\ve)$ as $n\to\infty$ and $\ve\to 0$ (in this order).

\begin{defi}[Energy function] 
  Let $X$ be a general metric space and let $\mu$ be a Borel probability measure on $X$. For $q\in \R\setminus\{1\}$ and  $\ve\in(0,1)$, one defines the so-called \emph{energy function} of $\mu$ by the law
\begin{eqnarray}
\label{ef}
I_{\mu}(q,\ve)=\int_{\supp(\mu)}\mu(B(x,\ve))^{q-1}d\mu(x),
\end{eqnarray}
where $\supp(\mu)$ is the topological support of $\mu$.
\end{defi}
The next result shows that the two previous definitions are intimately related.

\begin{teo}[Pesin \cite{Pesin1993,Pesin1997}]
\label{Pesincorrelation}
Let $X$ be a Polish metric space, assume that $\mu$ is ergodic and let $q\in\N\setminus\{1\}$. Then, there exists a set $Z\subset X$ of full $\mu$-measure such that, for each $R,\eta>0$ and each $x\in Z$, there exists an $N=N(x,\eta,R)\in\N$ such that
\[|C_q(x,n,\varepsilon)-I_{\mu}(q,\varepsilon)|\leq \eta\]
holds for each $n\geq N$ and each $0<\varepsilon\leq R$. In other words, $C_q(x,n,\varepsilon)$ tends to $I_{\mu}(q,\varepsilon)$ when $n\to \infty$ for $\mu$-almost every $x\in X$, uniformly over $\varepsilon\in(0,R]$.
\end{teo}

Taking into account Theorem~\ref{Pesincorrelation}, it is natural to introduce the following dimensions.

\begin{defi}[Generalized fractal dimensions]\label{gfd}
 Let $X$ be a general metric space, let $\mu$ be a Borel probability measure on $X$, and let $q\in \mathbb{R}\setminus\{1\}$. The so-called upper and lower $q$-generalized fractal dimensions of $\mu$ are defined, respectively, as
$$
 D^+_{\mu}(q)=\limsup_{\varepsilon\downarrow 0} \frac{\log I_{\mu}(q,\varepsilon)}{(q-1)\log \varepsilon}  ~~\mbox{ and  }~~ D^-_{\mu}(q)=\liminf_{\varepsilon\downarrow 0} \frac{\log I_{\mu}(q,\varepsilon)}{(q-1)\log \varepsilon}.
$$
If the limit $\ve\to 0$ exists, we denote it by $D_{\mu}(q)$, the so-called \emph{$q$-generalized fractal dimension} (also known as $q$-Hentchel-Procaccia dimension). For $q=1$, one defines the so-called upper and lower entropy dimensions (see~\cite{Barbaroux} for a discussion about the connection between entropy dimensions and Rényi information dimensions), respectively, as
\[D^+_{\mu}(1)=\limsup_{\varepsilon\downarrow 0} \frac{\int_{\supp(\mu)} \log \mu(B(x,\ve))d\mu(x)}{\log \varepsilon},\]
\[ D^-_{\mu}(1)=\liminf_{\varepsilon\downarrow 0} \frac{\int_{\supp(\mu)} \log \mu(B(x,\ve))d\mu(x)}{\log \varepsilon}.\]
\end{defi}

\begin{defi}[lower and upper packing and Hausdorff dimensions of a measure \cite{Mattila}]\label{HPdim}
  Let $\mu$ be a positive Borel measure on $(X,\mathcal{B})$. The lower and upper packing and Hausdorff dimensions of $\mu$ are defined, respectively, as 
\begin{eqnarray*}
\dim_{K}^-(\mu)&=&\inf\{\dim_{K}(E)\mid \mu( E)>0, ~ E\in \mathcal{B}\},\\
\dim_{K}^+(\mu)&=&\inf\{\dim_{K}(E)\mid \mu(X\setminus E)=0, ~ E\in \mathcal{B}\},
\end{eqnarray*}
where $K$ stands for $H$ (Hausdorff) or $P$ (packing); here, $\dim_{H(P)}(E)$ represents the Hausdorff (packing) dimension of the Borel set $E$ (see~\cite{Mattila} for details).
\end{defi}

\begin{defi}[lower and upper local dimensions of a measure]\label{Locdim}
Let $\mu$ be a positive finite Borel measure on $X$. One defines the upper and lower local dimensions of $\mu$ at $x\in X$ as
$$\overline{d}_{\mu}(x)=\limsup_{\ve\to 0}\frac{\log \mu(B(x,\ve))}{\log \ve} ~~\mbox{ and }~~ \underline{d}_{\mu}(x)=\liminf_{\ve\to 0}\frac{\log \mu(B(x,\ve))}{\log \ve},$$ 
if, for every $\ve> 0$, $\mu(B(x;\ve))>0$; if not, $\overline{\underline{d}}_\mu(x):=+\infty$.
\end{defi}

Some useful relations involving the generalized, Hausdorff and packing dimensions of a pro-bability measure are given by the following inequalities, which combine Propositions~4.1 and~4.2 in~\cite{Barbaroux} with Proposition~1.1 in \cite{AS} (although Propositions~4.1 and~4.2 in~\cite{Barbaroux} were originally proved for probability measures defined on $\R$, one can extend them to probability measures defined on a general metric space $X$; see also \cite{Rudnicki}).

\begin{propo}[\cite{Barbaroux,Rudnicki}]
\label{BGT1} 
Let $\mu$ be a Borel probability measure over $X$, let $q>1$ and let $0<s<1$. Then,
\begin{eqnarray}\label{BGTeq}
  D^-_{\mu}(q)\leq \mu\textrm{-}\essinf \underline{d}_{\mu}(x)=\dim_H^-(\mu)\leq \mu\textrm{-}\esssup \underline{d}_{\mu}(x)=  \dim_H^+(\mu)\leq D^-_{\mu}(s),
\end{eqnarray} and
\begin{eqnarray}\label{BGTeq1}
D^+_{\mu}(q)\leq \mu\textrm{-}\essinf \overline{d}_{\mu}(x)=\dim_P^-(\mu)\leq  \mu\textrm{-}\esssup \overline{d}_{\mu}(x)= \dim_P^+(\mu)\leq D^+_{\mu}(s). 
\end{eqnarray}

Moreover, $D_{\mu}^{\pm}(q)\leq D_{\mu}^{\pm}(1)\leq D_{\mu}^{\pm}(s)$.
\end{propo}

Our first result gives, under some assumptions, 
upper and lower bounds for the upper and lower generalized fractal dimensions of a probability measure defined on a compact metric space.

\begin{teo}
\label{torigual0}
Let $X$ be a compact metric space, let $\mu$ be a probability Borel measure on $X$ and suppose that there exist constants $0<\alpha<\beta<+\infty$ such that, for each $x\in \supp(\mu)$, $\alpha\le\underline{d}_{\mu}(x)\le \overline{d}_{\mu}(x)\le\beta$. Then, for each $s<1$ and each $q>1$, one has
\[\alpha\leq D^-_{\mu}(q)\leq D_\mu^-(1)\le D_\mu^+(1)\leq D^+_{\mu}(s)\le\beta.\]  
\end{teo}

The following result also presents upper and lower bounds for the upper and lower generalized fractal dimensions of an invariant measure for which Brin-Katok's Theorem is satisfied punctually.
\begin{teo}
\label{BKpunctually}
Let $(X,f)$ be a topological dynamical system such that $X$ is a compact metric space, and let $\mu$ be an invariant measure. Suppose that Brin-Katok's Theorem is satisfied  punctually, and suppose also that $f$ is a continuous function for which there exist  constants $\Lambda,\lambda>1$ and $\delta>0$  such that, for each $x,y\in X$ so that $d(x,y)<\delta$,
\begin{equation}\label{BLIP}
  \lambda\,d(x,y)\leq d(f(x),f(y))\leq \Lambda\,d(x,y).
\end{equation}
Then, for each $s<1$ and each $q>1$, one has
\[\frac{h_{\mu}(f)}{\log \Lambda}\leq D^-_{\mu}(q)\leq D_\mu^-(1)\le D_\mu^+(1)\leq D^+_{\mu}(s)\leq \frac{h_{\mu}(f)}{\log \lambda}.\] 
\end{teo}

Invariant measures that satisfy the hypotheses of Theorem~\ref{BKpunctually} are, for example:
\begin{enumerate}
\item Any homogeneous measure of a topological dynamical system $(X,f)$ (if it exists) that satisfies~\eqref{BLIP} (see Definition~\ref{MU} and Remark~\ref{bk}).
\item The Gibbs measures of maximal entropy of subshifts, $\sigma$, of finite type (see~\cite{BarreiraSchmeling} for the details and related situations); it is possible to endow the space with a metric such that $\sigma$ satisfies the second inequality in~\eqref{BLIP}. Moreover, if $\sigma$ is a topologically mixing subshift with the specification property, then any $\sigma$-invariant measure of maximal entropy is a Gibbs measure.
\item The maximal entropy measures of expansive homeomorphisms with specification (they are related to the equilibrium measures of the cohomology class of $\varphi\equiv 0$; see \cite{TakensVerbitsky}); moreover, due to Theorem~\ref{expanshyperbolic}, it follows that $f$ satisfies the second inequality in~\eqref{BLIP}, where $d$ is the hyperbolic metric given by Theorem~\ref{expanshyperbolic}. This is particularly true for Axiom A systems (see Subsection~\ref{EXPHOMEO} for details).
 \item The invariant measure of maximal entropy and maximal dimension of an expanding map supported on a conformal repeller $J$. Namely, let us assume that $f$ is topologically mixing, and let $m$ be the unique equilibrium measure corresponding to the H\"older continuous function $-s\log|a(x)|$ on $M$, where $s$ is the unique root of Bowen's equation $P_J(-s\log|a|)=0$ (see Appendix II in~\cite{Pesin1997}). $m$ is the measure of maximal dimension (that is, $\dim_Hm=\dim_HJ=s$), and it satisfies the conditions of Theorem~\ref{BKpunctually}; moreover, it is already known that, for each $q\in\R$, $D_m^\pm(q)=s$; see Remark 1 on page 219 in~\cite{Pesin1997}.
\end{enumerate}

\subsection{ $f$-homogeneous measures}

\begin{defi}
\label{MU}
Let $f$ be a continuous transformation of a compact metric space $(X,d)$. A Borel probability measure $\mu$ on $X$ is said to be $f$-homogeneous if for each $\ve>0$, there exist $\delta>0$ and $c>0$ such that, for each $n\in \N$ and each $x,y \in X$,  
\begin{equation}\label{CMU}
  \mu(B(y,n,\delta))\leq c\,\mu(B(x,n,\ve)),
\end{equation}
where $B(x,n,\ve):=\{y\in X\mid d(f^i(x),f^j(y))<\ve,\,\forall i=0,\ldots,n\}$ is the Bowen ball of size $n$ and radius $\ve$, centered at $x$.
\end{defi}

The homogeneous measures are defined of general manner in \cite{Bowen1973}; here, we only consider the case of compact metric spaces. It follows directly from Definition~\ref{MU} that if $\mu$ is a non-trivial $f$-homogeneous measure, then $\supp(\mu)=X$ (namely, suppose that there exist $x\in X$ and $\ve>0$ such that $\mu(B(x,\ve))=0$; then, for each $n\in\mathbb{N}$, $\mu(B(x,n,\ve))=0$. It follows now from this condition and~\eqref{CMU} that, for each $y\in X$, $\mu(B(y,1,\delta))=0$, and finally, from the continuity of $f$, that for each $y\in X$, there exists a $\delta(y)>0$ such that $\mu(B(y,\delta(y)))=0$. Since $X$ is compact, one concludes that $\mu\equiv 0$). 

Simple examples of homogeneous measures are the Lebesgue measure, invariant under the Arnold-Thom cat map, and the Bowen-Margulis measure for Axiom A diffeomorphisms (see 
Proposition 19.7 in \cite{Sigmundlibro}). The latter is defined as follows: if $\delta_x$ denotes the Dirac measure with support $\{x\}$,  consider the $f-$invariant measure 
\[\mu_n:= \sum_{x\in \fix(f^n)} \frac{\delta_x}{\card(\fix(f^n))},\]
where $\fix(f)=\{x\in X\mid f(x)=x\}$. By the weak compactness of $\M(f)$ (the space of $f$-invariant probability measures, endowed with the topology of the weak convergence) and the specification property, this sequence has at least one ergodic weak accumulation point $\mu$ (the Bowen-Margulis measure) with maximal entropy, i.e. $h_{\mu}(f)=h(f)$ (here, $h(f)$ stands for the topological entropy of $f$; see~\eqref{topentropy1}). Note $\mu$ is non-atomic, ergodic and $\supp(\mu)=X$; see \cite{Sigmundlibro} for more details.

\begin{rek}
\label{bk}
We note that Brin-Katok's Theorem is punctually satisfied for $f$-homogeneous measures: one has, for each $x\in X$,
\begin{eqnarray*}
\lim_{\ve\to 0}\liminf_{n\to 0}\frac{-\log\mu(B(x,n,\ve))}{n}= \lim_{\ve\to 0}\limsup_{n\to 0}\frac{-\log\mu(B(x,n,\ve))}{n}=h_{\mu}(f).
\end{eqnarray*}
\end{rek}
\begin{proof}
By the definition of a homogeneous measure, for each $\ve>0$, there exist $0 <
\delta(\ve) < \ve$ and $c > 0$ such that, for each $n\in \N$ and each $x,y\in X$,
\[\mu(B(y,n,\delta(\ve)))\leq c\,\mu(B(x,n,\ve)).\]
Thus,
\[\lim_{\ve\to 0}\limsup(\inf)_{n\to \infty}-\frac{1}{n}\log \mu(B(y,n,\delta(\ve)))\geq \lim_{\ve\to 0}\limsup(\inf)_{n\to \infty}-\frac{1}{n}\log \mu(B(x,n,\ve)).\]
Analogously, given $\tilde{\ve}:=\delta(\ve)>0$, there exist $0 < \tilde{\delta}(\tilde{\ve}) < \tilde{\ve}$ and $\tilde{c} > 0$ such that
\[\lim_{\tilde{\ve}\to 0}\limsup(\inf)_{n\to \infty}-\frac{1}{n}\log \mu(B(x,n,\tilde{\delta}(\tilde{\ve})))\geq \lim_{\tilde{\ve}\to 0}\limsup(\inf)_{n\to \infty}-\frac{1}{n}\log \mu(B(y,n,\tilde{\ve})).\]

This proves that the limits do not depend on $x\in X$. The result follows now from Brin-Katok's Theorem.
\end{proof}

The next result is a direct consequence of Theorem~\ref{BKpunctually} and Remark~\ref{bk}.

\begin{cor}
\label{corigual}
Let $(X,f,\mu)$ be a dynamical system such that $\mu$ is an $f$-homogeneous measure and $f$ is a function which satisfies the hypothesis of Theorem~\ref{BKpunctually}. Then, for each $s<1$ and each $q>1$, one has
\[\frac{h_{\mu}(f)}{\log \Lambda}\leq D^-_{\mu}(q)\leq D_\mu^-(1)\le D_\mu^+(1)\leq D^+_{\mu}(s)\leq \frac{h_{\mu}(f)}{\log \lambda}.\]   
\end{cor}

The next result, which is already known in the literature (see Theorem 2.5 in~\cite{Simpelaere} and \cite{Pesin1997}), is an extension of Young's formula (\cite{Young}) to the generalized fractal dimensions of the Bowen-Margulis measure associated to a $C^{1+\alpha}$-Axiom A system over a two-dimensional compact Riemannian manifold.

\begin{cor}
\label{quasiyoung}
Let $T:X\rightarrow X$ be a $C^{1+\alpha}$-Axiom A system ($\alpha>0$) over a two-dimensional compact Riemannian manifold $M$. Let $\mu$ be its Bowen-Margulis measure and let $\lambda_1(\mu)\geq\lambda_2(\mu)$ be its Lyapounov exponents. Then, for each $q\in\mathbb{R}$,
\begin{eqnarray}
\label{igualdad}
D^+_{\mu}(q)= D^-_{\mu}(q)=h_{\mu}(T)\left[ \frac{1}{\lambda_1}-\frac{1}{\lambda_2}\right].
\end{eqnarray}
\end{cor}

We present an alternative proof of this result which is directly based on the proof of Young's Theorem and Theorem~\ref{torigual0} (see Section~\ref{AXIOM-A}).

\subsection{Expansive homeomorphisms}
\label{EXPHOMEO}

We are also interested in dimensional properties of invariant measures for expansive homeomorphisms.
\begin{defi}
\label{homeoexpansive}
Let $X$ be a metrizable space, and let $f:X \rightarrow X$ be a homeomorphism. $f$ is said to be expansive if there exists a $\delta > 0$ such that, for each pair of different points $x, y \in X$, there exists an $n\in\Z$ such that $d(f^n(x),f^n(y)) > \delta$, where $d$ is any metric which induces the topology of $X$. 
\end{defi}

Note that expansivity is a topological notion, i.e., it does not depend on the choice of a particular (compatible) metric under consideration, although the expansivity constant $\delta$ may depend on $d$.

Examples of expansive homeomorphisms are: Axiom A systems (see \cite{Sigmundlibro}), homeomorphisms that admit a Lyapunov function (see \cite{Lewowicz}), examples 1 and 2 in~\cite{Wine}, the shift system with finite alphabet, pseudo-Anosov homeomorphisms, quasi-Anosov diffeomorphisms, etc.

The following result shows that if $X$ is a compact metrizable space, then a homeomorphism $f:X\rightarrow X$ is expansive if $X$ admits a hyperbolic metric (the converse of this statement is also true; see Theorem 5.3 in \cite{Fathi1989}).

\begin{teo}[Theorem 5.1 in \cite{Fathi1989}]
\label{expanshyperbolic}
If $f:X\rightarrow X$ is an expansive homeomorphism over the compact metrizable space $X$, then there exist a metric $d$ on $X$, compatible with its topology, and numbers $k>1$, $\ve>0$ such that, for each $x,y\in X$, 
\begin{eqnarray}
\label{expansivemetric}
\max\{d(f(x),f(y)), d(f^{-1}(x), f^{-1}(y))\}\geq \min\{k\,d(x,y),\ve\}.
\end{eqnarray}
 Moreover, both $f$ and $f^{-1}$ are Lipschitz for $d$. The metric $d$ is called a hyperbolic metric for $X$.
\end{teo}

Aside from the results stated in Theorems~\ref{torigual0} and~\ref{BKpunctually}, we also have some estimates for the generalized fractal dimensions of invariant measures of expansive homeomorphisms (with respect to the hyperbolic metric given by Theorem~\ref{expanshyperbolic}) in terms of the metric and the topological entropies.

\begin{teo}
\label{allzerodimensions}
Let $f:X\rightarrow X$ be an expansive homeomorphism over a compact metric space $X$, and let $d$ be the respective hyperbolic metric. Then, for each invariant measure $\mu\in\M(f)$ and each $q\in[0,1)$, one has $D^{+}_{\mu}(q)\leq \frac{2h(f)}{\log k}$, where $k$ is defined in the statement of Theorem~\ref{expanshyperbolic}.
\end{teo}

\begin{rek} One should compare Proposition~\ref{allzerodimensions} with Theorem~5.4 in~\cite{Fathi1989}.
\end{rek}

\begin{teo}
\label{allzerodimensions1}
Let $f:X\rightarrow X$ be an expansive homeomorphism over a compact metric space $X$, and let $d$ be the respective hyperbolic metric. Then, for each invariant measure $\mu\in\M(f)$ and each $q\ge 1$, one has $D^{+}_{\mu}(q)\leq h_\mu(f)\log k$, where $k$ is defined in the statement of Theorem~\ref{expanshyperbolic}.
\end{teo}

Let $T:X\rightarrow X$ be a $C^{1}$-Axiom A system over a two-dimensional compact Riemannian manifold $M$. It is known that such transformation is an expansive homeomorphism (see \cite{Sigmundlibro}).

Let $\mathcal{M}(T)$ be the space of all $T$-invariant probability measures,  endowed with the weak topology (that is the coarsest topology for which the net $\{\mu_\alpha\}$ converges to $\mu$ if, and only if, for each bounded and continuous function $\varphi$, $\int \varphi d\mu_\alpha\rightarrow \int \varphi d\mu$). Theorem~6 in~\cite{Sigmund1970} states that $\{\mu\in\M(T)\mid h_\mu(T)=0\}$ is a residual subset of $\M(T)$. The next result is a direct consequence of this fact and Theorem~\ref{allzerodimensions1}.

\begin{teo}
\label{zerocorrandhausdorff}
Let $T:X\rightarrow X$ be a $C^1$-Axiom A, and let $q\geq 1$. Then, the set $CD_0:=\{\mu\in \M(T)\mid D_{\mu}^+(q)=0$, under the hyperbolic metric $d\}$ is residual in $\M(T)$.
\end{teo}

Theorems \ref{Pesincorrelation} and \ref{zerocorrandhausdorff} may be combined with Proposition~\ref{BGT1} in order to produce the following result. Let $q\in\mathbb{N}\setminus\{1\}$; if $\mu\in CD_0\cap\M_e$ (where $\M_e:=\{\mu\in\M(T)\mid\mu$ is ergodic$\}$; see~\cite{Sigmund1970} for a proof that this set is generic), then there exists a Borel set $Z\subset X$, $\mu(Z)=1$, such that for each $x\in Z$, one has $\overline{\alpha}_q(x)= D_{\mu}^+(q)=0$.

This means that for each $x\in Z$ and each $\alpha>0$, there exists a $\delta=\delta(x,\alpha)>0$ such that if $0<\ve<\delta$, then there exists an $N=N(x,\alpha,\delta)\in \N$ such that, for each $n>N$, one has $C_q(x,n,\ve)\geq \ve^{(q-1)\alpha}$. Thus, one has $\gamma=\card\,\{(i_1\cdots i_q)\in \{0,1,\cdots, n\}^q\mid d(T^{i_j}(x),T^{i_l}(x))\leq \ve$ for each $j,l=0,\ldots,q\}\geq \ve^{(q-1)\alpha} \, n^q$ (recall that $d$ is the hyperbolic metric of $X$), which means that $\gamma$ is of order $n^q$ for $n$ large enough. The conclusion is that the orbit of a typical point (with respect to $\mu$) is very ``tight'' (it is some sense, similar to a periodic orbit).

The next result is a direct consequence of the proof of Theorem~\ref{zerocorrandhausdorff}.

\begin{cor}
  Let $X$ be a compact metric space, let $f:X\rightarrow X$ be an expansive homeomorphism  and let $q\ge 1$. If there exists  $\mu\in\mathcal{M}(f)$ such that $D^{+}_{\mu}(q)>0$ (with respect to a hyperbolic metric), then $h(f)\geq h_{\mu}(f)>0$.
\end{cor}

Since each $C^{1}$-Axiom A system over a compact smooth manifold $M$ is an expansive homeomorphism (see \cite{Sigmundlibro}), it is easy to check that the usual metric in $M$ satisfies~\eqref{expansivemetric}.

\begin{cor}
\label{entropymax}
 Let $f:X\rightarrow X$ be an expansive homeomorphism with specification over a compact metric space $X$, and let $\mu\in\M(f)$ be a measure of maximal entropy (that is, $h_\mu(f)=h(f)$).  Then, for each $q>1$, one has
\[\frac{h(f)}{\log \Lambda} \leq D^-_{\mu}(q), \]
where $\Lambda>1$ is the Lipschitz constant for $f$ (under the hyperbolic metric). 
\end{cor}
\begin{proof} 
  Given that $\{x\in X\mid \overline{h}_\mu(f,x)=\underline{h}_\mu(f,x)=h_\mu(f)\}=X$ for each $\mu\in\M(f)$ of maximal entropy (see Remark 9.A in~\cite{TakensVerbitsky}), the result follows from Theorem~\ref{BKpunctually}.
\end{proof}

\subsection{Organization}

The paper is organized as follows. In Section~\ref{boundedcorr}, we present the proofs of Theorems~\ref{torigual0} and ~\ref{BKpunctually}. Sections~\ref{AXIOM-A} and~\ref{zerocorrgeneric} 
are devoted, respectively, to the proofs of Corollary~\ref{quasiyoung}, Theorems~\ref{allzerodimensions} and~\ref{allzerodimensions1}. 
Finally, in Section~\ref{secentropy}, we show that, for some dynamical systems, the metric entropy of an invariant measure is typically zero, settling a conjecture posed by Sigmund in~\cite{Sigmund1974} for Lipschitz transformations which satisfy the specification property.





\section{Proofs of Theorems~\ref{torigual0} and~\ref{BKpunctually}}
\label{boundedcorr}

Before we present the proof of Theorem~\ref{torigual0}, some preparation is required; the strategy adopted here is inspired by \cite{Afraimovich,Saussol,Young}. Let, for each Borel probability measure $\mu$ and each $x\in X$,
\begin{eqnarray*}
\underline{d}_{\mu,i}(x):=\liminf_{\ve\to 0} \inf_{y\in \tilde{B}(x,\ve)}\frac{\log\mu(B(y,\ve))}{\log \ve}
\end{eqnarray*}
and 
\begin{eqnarray*}
\overline{d}_{\mu,s}(x):=\limsup_{\ve\to 0} \sup_{y\in \tilde{B}(x,\ve)}\frac{\log\mu(B(y,\ve))}{\log \ve}
\end{eqnarray*}  
be the so-called lower and upper uniform local dimensions of $\mu$ at $x$, where $\tilde{B}(x,\ve)= B(x,\ve)\cap \supp\mu$. It is straightforward to prove that the local uniform dimensions of a Borel probability measure coincide with its ordinary local dimensions (if $x\notin\supp\mu$, then $\overline{\underline{d}}_{\mu}(x):=\infty$ also coincide with the local uniform dimensions). 

\begin{proof01}
Since the arguments used in the proof of the first and the last inequalities are similar, we just present the proof that, for each $q>1$, $D^-_{\mu}(q)\geq \alpha$. The second and the fourth inequalities come, then, from Proposition~\ref{BGT1}.

Fix $q>1$, let $x\in\supp(\mu)$, and let $\eta>0$; then, there exists an $\ve(x)>0$ such that, for each $\ve\in(0,\ve(x))$ and each $y\in B(x,\ve)$,
\[\frac{\log\mu(B(y,\ve))}{\log \ve} \geq \inf_{y\in \tilde{B}(x,\ve)}\frac{\log\mu(B(y,\ve))}{\log \ve}\geq \alpha-\eta.\]
Thus, for each $x\in\supp(\mu)$ and each $\eta>0$, there exists an $\ve(x)>0$ such that, for each $\ve\in(0,\ve(x))$ and each $y\in B(x,\ve)$,
\begin{equation}\label{estimation}
  \mu(B(y,\ve))\leq \ve^{\alpha-\eta}. 
\end{equation}

Now,  since $\{B(x,\ve(x))\}_{x\in\supp(\mu)}$ is an open covering of the compact set $\supp(\mu)$, there exists a finite sub-family of $\{B(x,\ve(x))\}_{x\in \supp(\mu)}$ which also  covers $\supp(\mu)$. Let $\{B(x_i,\ve(x_i))\}_{i=1}^{k}$ be this sub-covering and let $\ve(k):=\min\{\ve(x_1), \ldots, \ve(x_k)\}$.
 
Consider  the following (finite) covering of $\supp(\mu)$ by balls of radius $\ve(k)$:
\[\supp(\mu) \subset \bigcup_{j=1}^{N}B(y_{j},\ve(k)),\] 
where $y_j\in  \overline{B}(x_l,\ve(x_l))$ for some $l\in\{1,\ldots, k\}$ (note that since, for each $l\in\{1,\ldots, k\}$, $\overline{B}(x_l,\ve(x_l))$ is compact, the open covering $\{B(y,\ve(k))\}_{\{y\in\overline{B}(x_l,\ve(x_l))\}}$ of $\overline{B}(x_l,\ve(x_l))$ admits a finite sub-covering). Now, let $\{ A_j\}_{j=1}^{M}$ be the disjoint covering of $\supp(\mu)$ obtained by removing the self-intersections of the elements of the previous covering; then, 
\begin{eqnarray}
\label{partition}
\supp(\mu)= \biguplus_{j=1}^{M} A_j\cap \supp(\mu). 
\end{eqnarray}
Fix $j\in\{1,\ldots,M\}$ and let $y\in A_j\cap \supp(\mu)$; there exists an $l\in\{1,\ldots,k\}$ such that $y\in B(x_l,\ve(x_l))\cap \supp(\mu)$. It follows from (\ref{estimation}) that, for each $0<\ve\leq \ve(k)\leq \ve(x_i)$, one has 
\begin{eqnarray*}
\mu(B(y,\ve))\leq \ve^{\alpha-\eta}.
\end{eqnarray*}
Therefore, 
\begin{eqnarray}
\label{clave3}
\nonumber\int_{A_j}\mu(B(y,\ve))^{q-1}d\mu(y)&=&\int_{A_j\cap \supp\mu}\mu(B(y,\ve))^{q-1}d\mu(y)\\
&\leq& \int_{A_j\cap \supp\mu}\ve^{(q-1)(\alpha-\eta)}d\mu(y)= \ve^{(q-1)(\alpha-\eta)}\,\mu(A_j).
\end{eqnarray}
Now, by (\ref{partition}) and (\ref{clave3}), one gets
\begin{eqnarray*}
\label{clave4}
\int_{\supp(\mu)}\mu(B(y,\ve))^{q-1}d\mu(y)
&=& \int_{\biguplus_{j=1}^{M} A_j\cap \supp\mu}\mu(B(y,\ve))^{q-1}d\mu(y)\\
&=& \sum_{j=1}^{M}\int_{A_j\cap \supp\mu}\mu(B(y,\ve))^{q-1}d\mu(y)\\
&\leq & \sum_{j=1}^{M} \ve^{(q-1)(\alpha-\eta)} \mu(A_j)\\ 
&= & \ve^{(q-1)(\alpha-\eta)}.
\end{eqnarray*}
Thus,
\[D^-_{\mu}(q)= \liminf_{\ve\to 0}\frac{\log\int_{\supp(\mu)}\mu(B(y,\ve))^{q-1}d\mu(y)}{(q-1)\log \ve}\geq \alpha -\eta.\]
The result now follows, since $\eta>0$ is arbitrary.
\end{proof01}


In order to prove Theorem~\ref{BKpunctually}, we need to prove some inequalities relating the local uniform dimensions and the local upper and lower entropies of invariant measures. This result is also used in the discussion involving the typical value of the metric entropy of an invariant measure for some particular dynamical systems (see Section~\ref{secentropy}).

\begin{lema}
\label{infsup}
Let $(X,f)$ be a topological dynamical system such that $X$ is a Polish metric space, and let $\mu\in\M(f)$. 
\begin{itemize}
\item[\emph{i)}]
If $f$ is a continuous function for which there exist  constants $\Lambda>1$ and $\delta>0$  such that, for each $x,y\in X$ so that $d(x,y)<\delta$, $d(f(x),f(y))\leq \Lambda\,d(x,y)$, then for each $x\in X$, 
\begin{eqnarray}\label{BKI}
\underline{d}_{\mu}(x)\geq \frac{\underline{h}_{\mu}(f,x)}{\log\Lambda}.
\end{eqnarray}
 Moreover, if $\mu\in\M_e(f)$, it follows that
\begin{eqnarray}\label{DH}
\dim_H^-(\mu)\geq \frac{h_{\mu}(f)}{\log\Lambda}.
\end{eqnarray}
\item[\emph{ii)}] If $f$ is a continuous function for which if there exist  constants $\lambda>1$ and $\delta>0$  such that, for each $x,y\in X$ so that $d(x,y)<\delta$,  $\lambda\,d(x,y)\leq d(f(x),f(y))$, then for each $x\in X$,
\begin{eqnarray}\label{BKS}
\overline{d}_{\mu}(x)\leq \frac{\overline{h}_{\mu}(f,x)}{\log\lambda}.
\end{eqnarray}  
Moreover, if $X$ is compact and $\mu\in\M_e(f)$, it follows that 
\begin{eqnarray}\label{DP}
\dim_P^+(\mu)\leq \frac{h_{\mu}(f)}{\log\lambda}.
\end{eqnarray}  
Here,  $\overline{\underline{h}}_{\mu}(f,x):=\lim_{\ve\to 0} \limsup(\inf)_{n\to \infty} \displaystyle\frac{-\log\mu(B(x,n,\ve))}{n}$ is the upper (lower) local entropy of $(f,\mu)$ at $x\in X$. 
\end{itemize} 
\end{lema}

\begin{proof}
\noindent {\bf\mbox{i)}} Claim 1. One has, for each $x\in X$, each $n\in\mathbb{N}$ and each $0<\ve\le\min\{1/2, \delta/2\}$, $B(x,\ve \Lambda^{-n})\subset B(x,n,\ve)$, where $B(x,n,\ve):=\{ y\in X\mid d(f^i(x),f^i(y))<\ve,~ \forall~ i=0,\ldots,n\}$ is the Bowen ball of size $n$ and radius $\ve$, centered at $x$. Namely, fix $x\in X$, $n\in\mathbb{N}$ and $0<\ve\le\min\{1/2, \delta/2\}$, and let $y\in B(x,\ve \Lambda^{-n})$; then, since $\ve \Lambda^{-n}<\delta$, one has, for each $i=0,\ldots,n$, $d(f^i(x),f^i(y))<\ve$, proving the claim.

Now, it follows from Claim 1 that, for each $x\in X$ and each $0<\ve\le\min\{1/2, \delta/2\}$, 
\begin{eqnarray*}
\underline{d}_{\mu}(x)  &=&   \liminf_{n\to \infty} \frac{\log\mu(B(x,\ve\Lambda^{-n}))}{\log\ve \Lambda^{-n}}  \\
&\geq & \liminf_{n\to \infty} \frac{\log\mu(B(x,n,\ve))}{-n} ~\frac{1}{\frac{-\log\ve}{n}+\log \Lambda}\\ 
&=& \liminf_{n\to \infty} \frac{\log\mu(B(x,n,\ve))}{-n} \frac{1}{\log\Lambda}.
\end{eqnarray*}
Thus, taking $\ve\to 0$ in both sides of the inequalities above, the result follows.

Now, if $\mu\in\M_e(f)$, it follows from Lemma~2.8 in~\cite{Riquelme} that $\underline{h}_{\mu}(f,x)=\mu\textrm{-}\essinf h_\mu(T,y)$ is valid for $\mu$-a.e. $x$, and then, by Theorem~2.9 in~\cite{Riquelme}, that $\underline{h}_{\mu}(f,x)\ge h_\mu(T)$ is also valid for $\mu$-a.e. $x$. Relation~\eqref{DH} is now a consequence of relation~\eqref{BKI}  and Definition~\ref{HPdim}.

\vspace{0.5cm} \noindent {\bf\mbox{ii)}} Claim 2. One has, for each $x\in X$, each $n\in\mathbb{N}$ and each $0<\ve\le\delta$, $B(x,n,\ve)\subset B(x,\ve \lambda^{-n})$. Namely, fix $x\in X$, $n\geq1$ and $0<\ve\le\delta$, and let $y\in B(x,n,\ve)$ so that, for each $j=0,\ldots,n$, $d(f^{j}(x),f^j(y))<\ve\le\delta$; it follows from the hypothesis that $\lambda^n\,d(x,y)\leq d(f^n(x),f^n(y))<\ve$, and therefore that $d(x,y)<\ve \lambda^{-n}.$

Now, it follows from Claim 2 that, for each $x\in X$ and each $0<\ve\le\delta$, 
\begin{eqnarray*}
\limsup_{n\to \infty} \frac{\log\mu(B(x,n,\ve))}{-n} \frac{1}{\log\lambda} 
&= &  \limsup_{n\to \infty} \frac{\log\mu(B(x,n,\ve))}{-n} ~\frac{1}{\frac{-\log\ve}{n}+\log \lambda}\\
&\geq &  \limsup_{n\to \infty} \frac{\log\mu(B(x, \ve \lambda^{-n}))}{\log\ve\lambda^{-n}}\\
&=& \overline{d}_{\mu}(x).
\end{eqnarray*}
Thus, taking $\ve\to 0$ in both side of the inequalities above, the result follows.

Now, if $\mu\in\M_e(f)$, it follows from Brin-Katok's Theorem that, for $\mu$-a.e. $x\in X$, $\overline{h}_{\mu}(f,x)=\underline{h}_\mu(f,x)=h_\mu(f)$. Relation~\eqref{DP} is now a consequence of relation~\eqref{BKS} and Definition~\ref{HPdim}. 
\end{proof}


\begin{rek} It follows from Theorem~2.10 in~\cite{Riquelme} that if $X$ is a complete (non-compact) Riemannian manifold and $\mu\in\M_e(f)$, then~\eqref{DP} is also valid. 
\end{rek}

\begin{proof03}
It follows from Lemma~\ref{infsup} and the fact that Brin-Katok's Theorem is satisfied punctually that, for each $x\in X$, 
\begin{equation}\label{desentr}
  \frac{h_{\mu}(f)}{\log\Lambda}\leq \underline{d}_{\mu}(x)\le \overline{d}_{\mu}(x)\leq\frac{h_\mu(f)}{\log\lambda}.
    \end{equation}
The result is now a consequence of Theorem~\ref{torigual0}.
\end{proof03}

\section{Proof of Corollary~\ref{quasiyoung}}
\label{AXIOM-A}
\begin{proof02}~

 \textit{Claim 1.}  For each $x\in X$, one has $\underline{d}_{\mu}(x)\geq h_{\mu}(T)\left[ \frac{1}{\lambda_1}-\frac{1}{\lambda_2}\right]$.

 \noindent We follow the proof of part 1 of Lemma 3.2 in \cite{Young}.  Namely, let 
\begin{eqnarray*}
\Lambda =\{x\in M\mid x && \hspace{-0.6cm}\mbox{ is  regular in the sense of Oseledec-Pesin} \\
 &&\mbox{ and } \lim_{\ve \to 0}\liminf_{n_1,\, n_2 \to \infty} \frac{-\log \mu(B(x,n_1,n_2,\ve))}{n_1+ n_2}= h_{\mu}(T)\},
\end{eqnarray*}
where $B(x,n_1,n_2,\ve):=\{y\in X\mid d(T^jx,T^jy)<\ve,\,j=-n_2,\ldots,n_2\}$ is the bilateral Bowen ball of size $n_1+n_2+1$ and radius $\ve$.
  
   Since $\mu$ is an $f$-homogeneus measure and $T$ is a uniform hyperbolic transformation (note that the discussion presented in Remark~\ref{bk} can be adapted to bilateral Bowen balls), it follows that $\Lambda=M$ (see \cite{Sigmund1970,Sigmund1972}). Let $\chi_i=e^{\lambda_i}$. For each $x\in M$ and each $\ve>0$, it is straightforward to show (as in~\cite{Young}) that 
\begin{eqnarray*}
\underline{d}_{\mu}(x)\geq (h_{\mu}(T)-\ve)\left[ \frac{1}{\log\frac{\chi_1+2\ve}{1-\ve}}+\frac{1}{\log\frac{\chi_2^{-1}+2\ve}{1-\ve}}\right];
\end{eqnarray*}
indeed, it is possible to show that, for each $\rho>0$ and each $y\in B(x,K(x)^{-1}\rho/2)$, one has
\begin{eqnarray*}
\label{subset1}
B(y, K(x)^{-1}\rho/2)\subset B(x, K(x)^{-1}\rho)\subset B(x,n_1(\rho),n_2(\rho),\rho),
\end{eqnarray*}
where $K(x):M\to \R$ is a function that relates the distance in the $x$-chart and the Riemannian metric on $M$ by the formula $\|\cdot - \cdot\|_x\leq K(x) d(\cdot,\cdot)$. Since $\ve>0$ is arbitrary, the claim follows. 

 \textit{Claim 2.}  For each $x\in X$, one has $\overline{d}_{\mu}(x)\leq h_{\mu}(T)\left[ \frac{1}{\lambda_1}-\frac{1}{\lambda_2}\right]$.

 \noindent We follow the proof of part 2 of Lemma 3.2 in \cite{Young}. Namely, since $X$ is a uniformly hyperbolic set, one may define $\phi:X\rightarrow\mathbb{R}$ by the law $$\phi(x)=\phi=A_1K_1\min\{(\chi_1+2\ve)^{-1}),(\chi_2^{-1}+2\ve)^{-1}\},$$ where $A_1:=\inf_{x\in X}A(x)>0$ and $K_1:=\sup_{x\in X} K(x)<\infty$ (see the proof of Lemma 3.2 in~\cite{Young} for details).

 Now, since $\mu$ is $f$-homogeneous, it follows from Man\~e's estimate that, for each $x\in X$,
 \[\limsup_{n_1,n_2\to\infty}-\frac{1}{n_1+n_2}\log(\mu(B(x,n_1,n_2,\phi)))\le h_\mu(T).\]
 
 The rest of the proof follows the same steps presented in the proof of Lemma 3.2 in~\cite{Young}, taking into account that $\Lambda_1=X$. 

 \
 
The result follows now from Claims 1, 2, Theorem~\ref{torigual0} and Proposition~\ref{BGT1} (for the case $q=1$).
\end{proof02}

\begin{rek}
As in Theorem~4.4 in \cite{Young}, one has, for each $q\in\R$, that
  \[D^{\pm}_\mu(q)=\underline{C}(\mu)=\overline{C}(\mu)=\underline{C}_{L}(\mu)=\overline{C}_{L}(\mu)=\underline{R}(\mu)=\overline{R}(\mu)= h_{\mu}(T)\left[\frac{1}{\lambda_{1}}-\frac{1}{\lambda_{2}}\right],\]
where $\underline{C}(\mu), \overline{C}(\mu), \underline{C}_{L}(\mu),\overline{C}_{L}(\mu)$ are the capacities and $\underline{R}(\mu), \overline{R}(\mu)$ are the upper and lower Renyi dimensions of $\mu$. 
\end{rek}

\begin{rek}\label{RAA}
The Bowen-Margulis measure is an example of measure that does not belong to set $CD_0$ in Theorem~\ref{zerocorrandhausdorff}. In fact, one has $D_\mu(q)=\dim_H(\mu)=\dim_H(X)$ (which is equal 2 when $f$ is Anosov).
\end{rek}

\section{Proofs of Theorems~\ref{allzerodimensions} and~\ref{allzerodimensions1}}
\label{zerocorrgeneric}

 Let $(X,d)$ be a compact metric space, and let $f:X\to X$ be a homeomorphism. For each $n\in \N$, one defines a new metric $d_n$ on $X$ by the law
\[d_n(x,y)=\max\{d(f^k(x), f^k(y)): k=0, \cdots, n-1\}.\]
 Note that, for each $\ve>0$, the open ball of radius $\ve$ centered at $x\in X$ with respect to $d_n$ coincides with the Bowen dynamical ball of size $n$ and radius $\ve>0$, centered at $x$:
\[B(x,n,\ve)=\{y\in X: d_n(x,y)<\ve\}.\]

\begin{propo}
The metrics $d_n$ and $d$ induce the same topology on $X$.
\end{propo}
\begin{proof} This is a direct consequence of the fact that $f$ is a homeomorphism.
\end{proof}

Thus, for each $x\in X$, each $n\in\mathbb{N}$ and each $\ve>0$, $\overline{B}(x,n,\ve)=\{y\in X \mid d_n(x,y)\leq \ve\}$ is a closet set and $B(x,n,\ve)=\{y\in X \mid d_n(x,y)<\ve\}$ is an open set (both with respect to the topology induced by $d$).

Let $n\in \N$ and $\ve>0$. A subset $F$ of $X$ is said to be an $(n,\ve)$-generating set if, for each $x\in X$, there exists $y\in F$ such that $d_n(x,y)<\ve$.

Let $R(n,\ve)$ be the smallest cardinality of an $(n,\ve)$-generating set for $X$ with respect to $f$. Then, the following limit exists, and it is called the \emph{topological entropy} of $f$ (see~\cite{Walters}): 
\begin{eqnarray}
\label{topentropy1}
h(f) := \lim_{\ve\to 0} \limsup_{n\to \infty} \frac{1}{n} \log R(n,\ve)
=\lim_{\ve\to 0} \liminf_{n\to \infty} \frac{1}{n} \log R(n,\ve).
\end{eqnarray}

\begin{proof3}
It follows from Theorem \ref{expanshyperbolic} that there exist $k>1$ and $\ve>0$ such that, for each $x\in X$, each $n\in\mathbb{N}$ and each $0<r<\ve/k$, one has  $B(x,n,r)\subset B(x, k^{-n}r)$. Thus, $\mu(B(x,n,r))^{q-1}\geq \mu(B(x, k^{-n} r))^{q-1}$, and (taking $r$ sufficiently small so that $k^{-n}r<1$)
\begin{eqnarray}
\label{exp01}
\frac{\log \int_{\supp\mu} \mu(B(x, k^{-n} r))^{q-1}d\mu(x)}{(q-1) \log k^{-n} r}\leq \frac{\log \int_{\supp\mu} \mu(B(x,n,r))^{q-1}d\mu(x)}{(q-1) \log k^{-n} r}.
\end{eqnarray}

Let $E=\{x_i\}$ be an arbitrary $(2n+1,r/2)$-generating set (which, in particular, implies that $X\subset \bigcup_{x_i\in E} B(x_i,n,r/2)$). Since $X$ is compact, one may take $E$ as a finite subset of $X$. Let also $\tilde{F}=\{x_j\}$ be a subset of $E$ such that $\{B(x_j,n,r/2)\}_{\tilde{F}}$ is a covering of $\supp(\mu)$. Then, one has
\begin{eqnarray}
\label{exp02}
\int_{\supp(\mu)} \mu(B(x,n,r))^{q-1}d\mu(x) 
&\leq  &\sum_{x_j\in \tilde{F}} \int_{B(x_j,n,r/2)} \mu(B(x,n,r))^{q-1}d\mu(x) \nonumber\\
&\leq & \sum_{x_j\in \tilde{F}} \mu(B(x_j,n,r/2))^q \nonumber \\
&\leq &\sum_{x_i\in E} \mu(B(x_i,n,r/2))^q,
\end{eqnarray}
where we have used the fact that, for each $x\in B(x_j,n,r/2)$, $B(x_j,n,r/2)\subset B(x,n,r)$.

Now, by (\ref{exp01}) and (\ref{exp02}), one has 
\begin{eqnarray*}
\frac{\log \int_{\supp\mu} \mu(B(x, k^{-n} r))^{q-1}d\mu(x)}{(q-1) \log k^{-n} r} 
&\leq &\frac{\log \int_{\supp\mu} \mu(B(x,n,r))^{q-1}d\mu(x)}{(q-1) \log k^{-n} r}\\
&\leq & \frac{\log \sum_{x_i\in E}\mu(B(x_i,n,r/2))^q}{(q-1) \log k^{-n} r}.
\end{eqnarray*}
Thus, for $q=0$,   
\begin{eqnarray}
\label{eqq1}
\frac{\log \int_{\supp\mu} \mu(B(x, k^{-n} r))^{-1}d\mu(x)}{- \log k^{-n} r} \leq  \frac{\log R(n,r/2)}{ -\log k^{-n} r}=\frac{\log R(n,r/2)}{ n(\log k- \frac{\log r}{n})}.
\end{eqnarray}

Given that $D^{+}_{\mu}(q)$ is a decreasing function of $q$ (see \cite{Cutler1995} and \cite{Barbaroux}), one has $D^{+}_{\mu}(q) \leq D^{+}_{\mu}(0)$. Furthermore, since the function $\varphi:(0,\infty)\rightarrow\mathbb{R}$, $\varphi(\ve)=\int_{\supp\mu} \mu(B(x,\ve))^{-1}d\mu(x)$ is decreasing, it follows from Lemma~6.2 in \cite{Fathi1988} that
\begin{eqnarray}
\label{eqq2}
D^{+}_{\mu}(0)=\limsup_{n\to\infty}\frac{\log \int_{\supp\mu} \mu(B(x, k^{-n} r))^{-1}d\mu(x)}{- \log k^{-n} r}.
\end{eqnarray}   
Thus, it follows from 
(\ref{eqq1}) and (\ref{eqq2}) that
\begin{eqnarray*}
D^{+}_{\mu}(q) \leq D^{+}_{\mu}(0)=\limsup_{n\to\infty}\frac{\log \int_{\supp\mu} \mu(B(x, k^{-n} r))^{-1}d\mu(x)}{- \log k^{-n} r}
&\leq &  \limsup_{n\to\infty} \frac{\log R(n,r/2)}{ 2n} \frac{2}{\log k}.
\end{eqnarray*}
Therefore, taking $r\to 0$, the result follows from (\ref{topentropy1}).
\end{proof3}

\begin{cor}\label{zeroph}
  Let $(X,d,f)$ be as in the statement of Theorem~\ref{allzerodimensions}. If $h(f)=0$, then for each $\mu\in\M(f)$ and each $s\ge 0$, one has $D^{\pm}_{\mu}(s)=0$.  
\end{cor}

\begin{rek} It follows from Corollary~5.5 in~\cite{Fathi1989} that if a compact metric space admits an expansive homeomorphism whose topological entropy is zero, then its topological dimension is zero. See Section~3 in~\cite{Fathi1989} for examples of systems with zero topological entropy.
  \end{rek}

\begin{proof4}
It suffices, from Proposition~\ref{BGT1}, to prove the result for~$q=1$. It follows from Theorem \ref{expanshyperbolic} that there exist a hyperbolic metric $d$ which induces an equivalent topology on $X$, and numbers $k>1$, $\ve>0$ such that $f$ is expansive under this metric and, for each $0<r<\ve/k$ and each $x\in X$,  $B(x,n,r)\subset B(x, k^{-n}r)$. Thus, 
\begin{eqnarray}
\label{corrineq}
\frac{\int \log\mu(B(x, k^{-n} r))d\mu(x)}{\log k^{-n} r}\leq \frac{\int \log\mu(B(x,n,r))d\mu(x)}{\log k^{-n} r}.
\end{eqnarray}
\emph{Claim.} \[\limsup_{n\to\infty}\frac{ \int\log\mu(B(x,n,r))d\mu(x)}{\log k^{-n} r}\leq h_{\mu}(f) \log k.\]
Following the proof of Brin-Katok's Theorem, fix $r>0$ and consider a finite measurable partition $\xi$ such that $\diam \xi=\max_{C\in \xi} \diam(C)<r$. Let $\xi(x)$ be the element of $\xi$ such that $x\in\xi(x)$, and let $C_{n}^{\xi}(x)$ be the element of the  partition $\xi_n=\bigvee_{i=-n}^{n} f^{-i}\xi$ such that $x\in C_{n}^{\xi}(x)$. Given that $\xi(x)\subset B(x,r)$, one has 
\[C_{n}^{\xi}(x)=\bigcap_{i=-n}^{n}f^{-i}(\xi(f^ix))\subset \bigcap _{i=-n}^{n}f^{-i}(B(f^ix,r))=B(x,n,r),\]
from which follows that
\[\frac{ \int \log\mu(B(x,n,r))d\mu(x)}{-n}\leq \frac{ \int\log \mu(C_{n}^{\xi}(x))d\mu(x)}{-n}=\frac{H(\xi_n)}{n},\]
where $H(\xi_n)=-\sum_{C_{n}^{\xi}(x)\in \xi_n} \mu(C_{n}^{\xi}(x)) \log \mu(C_{n}^{\xi}(x))=\int -\log\mu(C_{n}^{\xi}(x))d\mu(x)$. Thus,
\begin{eqnarray*}
\limsup_{n\to\infty}\frac{ \int \log\mu(B(x,n,r))d\mu(x)}{-n}\leq \limsup_{n\to\infty}\frac{H(\xi_n)}{n}= H(f,\xi) \leq h_{\mu}(f),
\end{eqnarray*}
proving the claim.

Now, since for each $r>0$, each $k>1$ and each $n\in\N$, $\int \log\mu(B(x,k^{-n}r))d\mu(x)$ is finite (by~\eqref{corrineq} and Lemma~2.12 in~\cite{Verbitskiy}), it follows from an adaptation of Lemma~A.6 in~\cite{Moritz} that 
\begin{eqnarray}
\label{eq002}
\limsup_{r\to 0}\frac{\int\log\mu(B(x,r))d\mu(x)}{\log r}=\limsup_{n\to\infty}\frac{ \int \log\mu(B(x,k^{-n}r))d\mu(x)}{\log k^{-n}r}.
\end{eqnarray}

One concludes the proof of the proposition combining relations (\ref{corrineq}) and~\eqref{eq002} with Claim.
\end{proof4}

\begin{rek}
\label{RAA}
\begin{enumerate}
\item One can apply simultaneously Corollary~\ref{entropymax} and Theorem~\ref{allzerodimensions1} for measures of maximal entropy as follows: first, since $f$ is an expansive homeomorphism, it follows from Theorem~\ref{expanshyperbolic} that there exists an hyperbolic metric $d$ so that $f$ is Lipschitz; let $L>1$ be its Lipschitz constant and let $q>1$. One has, from Corollary~\ref{entropymax}, that $\frac{h(f)}{\log  L}\leq D^+_{\mu}(q)$,  and then, from Theorem~\ref{allzerodimensions1}, that $D^{+}_{\mu}(q)\leq h(f)\log k$.  Thus, $\frac{h(f)}{\log  L}\leq D^+_{\mu}(q)\leq h(f)\log k$.
  \item If $f$ is also expanding, with expanding constant $\lambda>1$, it follows from Theorem~\ref{BKpunctually} that for each $s<1$, $D^+_{\mu}(s)\leq \frac{h(f)}{\log \lambda}$. Now, it follows from Theorem~\ref{allzerodimensions} that  $D^{+}_{\mu}(s)\leq \frac{2h(f)}{\log k}$ (taking into account the hyperbolic metric $d$).  This shows the similarity of both results and a critical dependence on the metric.
\end{enumerate}
\end{rek}




\section{Generic sets of invariant measures with zero entropy}
\label{secentropy}

One can combine Lemma~\ref{infsup} with some results presented in~\cite{AS} in order to show that, in some situations, the set of invariant measures whose metric entropy is zero is residual.

In what follows, denote by $\M^{co}(f)$ the set of $f$-invariant periodic measures, that is, the set of measures of the form $\mu_x(\cdot):=\frac{1}{k_x}\sum_{i=0}
^{k-1}\delta_{f^i(x)}(\cdot)$, where $x\in X$ is an $f$-periodic point of period $k_x$, and $\delta_x(A)=1$ if $x\in A$ and zero otherwise.

\begin{teo}\label{GENENT}
Let $X$ be a Polish space and let $f:X\rightarrow X$ be an invertible transformation such that both $f$ and $f^{-1}$ are Lipschitz. Suppose that $\overline{\M^{co}(f)}=\M(f)$. Then,
  \[\{\mu\in\M(f)\mid h_\mu(f)=0\}\]
  is a residual subset of $\M(f)$.
  \end{teo}
\begin{proof} Firstly, we note that $\M_e(f)$ is a generic subset of $\M(f)$. Namely, the measures in $\M^{co}(f)$ are obviously ergodic. Hence, $\overline{\M_e(f)}=\M(f)$. Since one has from Theorem~2.1 in~\cite{Parthasarathy1961} 
  that $\M_e(f)$ is a $G_\delta$ subset of $\M(f)$, the result follows.

 Thus, one gets from Propositions~2.2 and~2.5 in~\cite{AS} that $\{\mu\in\M_e(f)\mid\dim_H^+(\mu)=0\}$ is a generic subset of $\M(f)$ (although Proposition~2.2 in~\cite{AS} was proven for the full-shift system presented in Subsection~1.2, the result can be extended to the dynamical system $(X,f)$ considered here). The result is now a consequence of Lemma~\ref{infsup}(i).
\end{proof}

\begin{cor}\label{GENENTFS} Let $(X,f,\mathcal{B})$ be the full-shift dynamical system over $X =\prod_{i=-\infty}^{+\infty}M$, where the alphabet $M$ is a Polish space. Then, 
\[\{\mu\in\M(f)\mid h_\mu(f)=0\}\]
is a residual subset of $\M(f)$.
\end{cor}

Corollary~\ref{GENENTFS} generalizes Theorem~1 in~\cite{Sigmund1971} (originally proved for $M=\mathbb{R}$) for any Polish space.

We also have a version of Theorem~\ref{GENENT} for topological dynamical systems.

\begin{teo}\label{GENENT1}
  Let $(X,f)$ be a topological dynamical system such that $f$ is Lipschitz, and suppose that $\overline{\M^{co}(f)}=\M(f)$. Then,
  \[\{\mu\in\M(f)\mid h_\mu(f)=0\}\]
  is a residual subset of $\M(f)$.
  \end{teo}
\begin{proof} Theorem~1.2 in~\cite{AS2} states that, for each $q\in(0,1)$, $\{\mu\in\M(f)\mid D^-_\mu(q)=0\}$ is a residual subset of $\M(f)$. The result is now a consequence of Proposition~\ref{BGT1} and Lemma~\ref{infsup}(i).
\end{proof}

Theorem~\ref{GENENT1} partially settles a conjecture posed by Sigmund in~\cite{Sigmund1974}, which states that if a topological dynamical system $(X,f)$ satisfies the specification property (and consequently, $\overline{\M^{co}(f)}=\M(f)$; see~\cite{Sigmund1974}), then $\{\mu\in\M(f)\mid h_\mu(f)=0\}$ is a residual subset of $\M(f)$.


\addcontentsline{toc}{chapter}{BIBLIOGRAPHY}
\bibliography{refs-1}

\end{document}